\documentclass[reqno]{amsart}
\usepackage{amsrefs}
\usepackage{amssymb}
\usepackage{calrsfs}
\usepackage{mathrsfs}
\usepackage{graphics,graphicx}
\usepackage{enumerate}
\usepackage{url}
\usepackage{xcolor}
\usepackage[ruled, lined, linesnumbered, longend]{algorithm2e}
\usepackage{hyperref}

\newtheorem{thm}{Theorem}[section]

\newtheorem{lemma}{Lemma}[section]

\newtheorem{remark}{Remark}[section]
\numberwithin{equation}{section}

\newcommand{\invert}[1]{\frac{1}{#1}}
\newcommand{\abs}[1]{\left\lvert#1\right\rvert}

\newcommand{\paranthesis}[1]{\left(#1\right)}
\newcommand{\llsum}{\mathlarger{\mathlarger{\sum}}}
\newcommand{\llprod}{\mathlarger{\mathlarger{\prod}}}
\newcommand{\llint}{\mathlarger{\mathlarger{\int}}}

\def\reals{\hbox{\rm I\kern-.18em R}}
\def\complexes{\hbox{\rm C\kern-.43em
\vrule depth 0ex height 1.4ex width .05em\kern.41em}}
\def\field{\hbox{\rm I\kern-.18em F}} 

\allowdisplaybreaks

\begin{document}

\title[Squares of primes and powers of two]{Representation of even integers as a sum of squares of primes and powers of two}

\author{Shehzad Hathi}
\address{School of Science, The University of New South Wales, Canberra, Australia}
\email{s.hathi@student.adfa.edu.au}

\date\today

\begin{abstract}
In 1951, Linnik proved the existence of a constant $K$ such that every sufficiently large even number is the sum of two primes and at most $K$ powers of 2. Since then, this style of approximation has been considered for problems similar to the Goldbach conjecture. One such problem is the representation of a sufficiently large even number as a sum of four squares of primes and at most $k$ powers of two. In 2014, Zhao proved this to be true with $k = 46$. In this paper, we reduce this to $k = 31$.

\end{abstract}

\maketitle

\section{Introduction} \label{sec:intro}

In the early 1950s, Linnik proved, first on the Generalised Riemann Hypothesis \cite{linnikprime1951} and then unconditionally \cite{linnikaddition1953}, that every sufficiently large even integer can be written as the sum of two primes and at most $K$ powers of two, where $K$ is an absolute constant. This approximates the Goldbach problem with the only addition being a very sparse set, $\mathscr{K}(x)$, of cardinality $O((\log x)^K)$ such that every sufficiently large even integer $N \le x$ can be written as the sum of two primes and some $n \in \mathscr{K}(x)$. Subsequently, in 1975, Gallagher \cite{gallagherprimes1975} significantly simplified Linnik's proof and paved the way for later results which provided explicit values for $K$. The best known value for $K$ under the Generalised Riemann Hypothesis is $K = 7$, proven independently by Heath-Brown and Puchta in \cite{heath-brownintegers2002} and Pintz and Ruzsa in \cite{pintzon2003}. Recently, Pintz and Ruzsa \cite{pintzon2020} also proved the best known unconditional value, $K = 8$.

Using the techniques from the original Linnik--Goldbach problem, a whole class of such problems has been looked at in recent years. One such problem is of representing a sufficiently large even integer as the sum of four squares of primes. While this is still an unsolved problem, given Lagrange's four squares theorem and some significant results obtained by Hua in the 1930s, it does not seem implausible. Hua \cite{huasome1938} showed that every sufficiently large integer congruent to 5 modulo 24 can be expressed as the sum of five squares of primes. Moreover, the number of integers $N \equiv 4 \mod 24$ and $N \le x$ that \textit{cannot} be represented as the sum of four squares of primes is $o(x)$. Note that some sort of congruence condition is necessary here because a prime $p \neq 2,3$ satisfies $p^2 \equiv 1 \mod 3$ and $p^2 \equiv 1 \mod 8$, and so, $p^2 \equiv 1 \mod 24$. Later, Br\"{u}dern and Fouvry \cite{brudernlagrange1994} also established that every sufficiently large integer $N \equiv 4 \mod 24$ can be represented as the sum of four squares of integers having at most 34 prime factors.

Motivated by these results, a Linnik-style approximation has been considered for this problem in the last two decades. More precisely, can we write every sufficiently large even integer as the sum of four squares of primes and a finite number of powers of two? This was proven to be true by Liu, Liu, and Zhan \cite{liusquares1999} in 1999. Since then, various explicit values (denoted by $k$) for the maximum number of powers of two required have been established:
\begin{align*}
    k &= 8330 &\text{(Liu \& Liu \cite{liurepresentation2000}),} \\
    k &= 165 &\text{(Liu \& L\"{u} \cite{liu2004four}),} \\
    k &= 151 &\text{(Li \cite{lifour2006}),} \\
    k &= 46 &\text{(Zhao \cite{zhao2014four}),} \\
    k &= 45 &\text{(Platt \& Trudgian \cite{platt2015linnik}).}
\end{align*}
In this paper, we prove the following result.

\begin{thm} \label{thm:main}
Every sufficiently large even integer can be represented as a sum of four squares of primes and at most 31 powers of two.
\end{thm}

Our method improves upon Zhao's work in \cite{zhao2014four} which relies on the circle method in combination with the linear sieve. In \S \ref{sec:outline}, we present the most important theoretical results required for our purpose from Zhao's work. Remark \ref{rem:kappahmistake} points out a calculation error in \cite{zhao2014four} which has a small impact on the results in \cite{zhao2014four}.

The crux of the circle method often lies in effectively bounding the contribution from minor arcs. In order to do so, we are interested in bounds on integrals of the type $\int_0^1 \abs{T(\alpha)^4 G(\alpha)^{2l}} d\alpha$, where $T(\alpha)$ and $G(\alpha)$ are as defined in (\ref{defn:expsums}). In \cites{liu2004four,lifour2006}, $l$ is chosen to be 2 whereas in \cite{zhao2014four}, $l$ is chosen to be 7. This method was motivated by the works of Wooley \cite{wooleyslim2002} and Tolev \cite{toleveon2005}. Here, by choosing $l = 5$ (as in Chapter 4 of \cite{zhao2012some}) and with more accurate numerical computations for $c_{0,l}$ (see (\ref{defn:c0})) outlined in \S \ref{sec:numcomp}, we are able to establish that $k = 31$ suffices. In \S \ref{sec:proofofmain}, we combine the theoretical results from \S \ref{sec:outline} and the computational work from \S \ref{sec:numcomp} to prove Theorem \ref{thm:main}. 

Apart from improvements in the computations, we also use a slightly better value of $\lambda$ in (\ref{defn:lambdaset}) from \cite{platt2015linnik}. This $\lambda$ was computed using a Pari/GP script\footnote{The relevant programs can be found here: \url{https://bit.ly/3IIxU6B}} first written for \cite{languasco2010on}. Table \ref{tab:l-k} lists the corresponding values of $k' (= k-2)$ for different choices of $l$. This table illustrates that there is a trade-off between minimising the contribution from the aforementioned integral and the factor (dependent on $\lambda$) multiplied with this integral in (\ref{ineq:minorarcs}). In this case, the choice $l = 5$ is optimal. In \S \ref{subsec:further}, we discuss ways to improve the result further.

The idea of considering integrals of the type $\int_0^1 \abs{T(\alpha)^4 G(\alpha)^{2l}} d\alpha$ for $l \ge 5$ and using a linear sieve has been employed in many Linnik--Goldbach problems recently. Some of these results rely on the value of $c_0$ from \cite{zhao2014four}*{Lemma 4.1} which has been improved here (see Lemma \ref{lem:c0}). In \S \ref{sec:applicatons}, we provide a non-exhaustive account of these problems. Also, as an application, we consider the closely related problem of representing an odd number as the sum of a prime and two squares of primes. This problem was first considered by Hua \cite{huasome1938} who proved that \textit{almost all} odd integers $n$ satisfying $n \neq 2 \mod 3$ can be written as the sum of a prime and two squares of primes. Again, a Linnik-style approximation to this problem was proven by Liu, Liu, and Zhan \cite{liusquares1999} who showed that every sufficiently large odd integer can be written as the sum of a prime, two squares of primes, and a finite number of powers of two. In later years, the following explicit values ($\mathfrak{K}$) for the maximum number of powers of two required were worked out:
\begin{align*}
    \mathfrak{K} &= 22000 &\text{(Liu \cite{liu2004representation}),} \\
    \mathfrak{K} &= 106 &\text{(Li \cite{li2007representation}),} \\
    \mathfrak{K} &= 83 &\text{(L\"u \& Sun \cite{lu2009integers}),} \\
    \mathfrak{K} &= 35 &\text{(Liu \cite{liu2014one}),} \\
    \mathfrak{K} &= 34 &\text{(Platt \& Trudgian \cite{platt2015linnik}),} \\
    \mathfrak{K} &= 31 &\text{(Liu \cite{liu2016two}),} \\
    \mathfrak{K} &= 17 &\text{(L\"u \cite{guangshi2018on}).}
\end{align*}
Here, we prove the following result in \S \ref{sec:applicatons}, improving upon Theorem 1.1 of \cite{guangshi2018on} which states that 17 powers of two suffice.

\begin{thm}
\label{thm:application}
Every sufficiently large odd integer can be written as a
sum of one prime, two squares of primes and at most 12 powers of two.
\end{thm}


In what follows, we will write $e^{2\pi in}$ as $e(n)$ and $\epsilon$ will denote an arbitrarily small positive real number.

\section{Theoretical framework} \label{sec:outline}

In this section, we present a few important results from \cite{zhao2014four} that are needed to prove Theorem \ref{thm:main}. In the process, we also clarify a few details and correct minor errors in \cite{zhao2014four}.

Let $N$ be a large even integer. To apply the circle method, we must first define the major arcs ($\mathfrak{M}$) and the minor arcs ($\mathfrak{m}$). For this purpose, let 
\begin{equation*}
    L = \frac{\log (N/\log N)}{\log 2}.
\end{equation*}
Set the parameters,
\begin{equation*}
    P = N^{1/5-\epsilon} \;\; \text{ and } \;\; Q = \frac{N}{P L^{2l}},
\end{equation*}
where $l$ is a positive integer that will be chosen later. The choice of $P$ here is consistent with that in \cites{zhao2014four,liu2004four} and is explained at the end of \S 3 in \cite{liu2003on}.

Let
\begin{equation*}
    \mathfrak{M}_{a,q} := \left\{ \alpha \,:\, \abs{\alpha-\frac{a}{q}} \le \invert{qQ} \right\}.
\end{equation*}
Then the major arcs and minor arcs are defined as:
\begin{equation*}
    \mathfrak{M} = \bigcup \limits_{1 \le q \le P} \bigcup \limits_{\substack{1 \le a \le q \\ (a,q) = 1}} \mathfrak{M}_{a,q} \;\; \text{ and } \;\; \mathfrak{m} = \left[ \invert{Q}, 1+\invert{Q} \right] \setminus \mathfrak{M}.
\end{equation*}

We also define an interval for the primes in our representations. For a small positive constant $\eta$, which, for convenience, we choose to be in $(0,10^{-10})$, let
\begin{equation*}
    \mathfrak{B} = \left[\sqrt{\paranthesis{1/4-\eta} N},\sqrt{\paranthesis{1/4+\eta}N}\right].
\end{equation*}
Then, for any positive integer $k$, the weighted number of representations
\begin{equation*}
    R_k(N) := \llsum_{\substack{{p_1}^2+{p_2}^2+{p_3}^2+{p_4}^2+2^{\nu_1}+\cdots+2^{\nu_k} = N \\ p_j \in \mathfrak{B} \; (1 \le j \le 4), \; 4 \le \nu_1,\ldots,\nu_k \le L}} \; \llprod_{j = 1}^4 \log p_j,
\end{equation*}
can be expressed as the sum of integrals over major arcs and minor arcs,
\begin{equation*}
    R_k(N) = \int_{\mathfrak{M}} T^4(\alpha) G^k(\alpha) e(-\alpha N) d \alpha + \int_{\mathfrak{m}} T^4(\alpha) G^k(\alpha) e(-\alpha N) d \alpha. 
\end{equation*}
with the following exponential sums:
\begin{equation} \label{defn:expsums}
    T(\alpha) := \sum_{p \in \mathfrak{B}} (\log p) e \paranthesis{p^2 \alpha} \;\; \text{ and } \;\; G(\alpha) := \sum_{4 \le \nu \le L} e \paranthesis{2^\nu \alpha}.
\end{equation}

Our objective is to show $R_k(N) > 0$ for sufficiently large $N$ and $k \ge 31$, which would imply Theorem \ref{thm:main}. To estimate the integral over major arcs, we have the following result from \cite{zhao2014four}*{p. 270} but valid for a larger range of $k'$.

\begin{lemma} \label{lem:majorarcs}
Let
\begin{equation*}
    \mathfrak{J}(h) = \llint_{-\infty}^{\infty} \paranthesis{\llint_{\sqrt{1/4-\eta}}^{\sqrt{1/4+\eta}} e \paranthesis{x^2 \beta} dx}^4 e(-h\beta) d\beta.
\end{equation*}
Then, for $k' \ge 19$ and $N \equiv 4 \mod 8$,
\begin{equation*}
    \int_{\mathfrak{M}} T^4(\alpha) G^{k'}(\alpha) e(-\alpha N) d \alpha \ge 0.9 \times 8 \mathfrak{J}(1) NL^{k'} + O \paranthesis{NL^{k'-1}}.
\end{equation*}
\end{lemma}
The factor 0.9 in Lemma \ref{lem:majorarcs} is due to numerical computations in \cite{zhao2014four}*{Lemma 4.4}. Upon redoing the computations, we find that the bound on max in the proof of Lemma 4.4 can be much tighter, that is, $\text{max} < 27.35$. This allows us to conclude Lemma \ref{lem:majorarcs} for $k' \ge 19$.

For the integral over the minor arcs, we will need a few definitions. Let
\begin{equation*}
    \mathcal{B}(p,h) := \sum_{a = 1}^{p-1} \abs{g(a;p)-1}^4 e(ah/p).
\end{equation*}
Here, $g(a;p)$ denotes the quadratic Gauss sum modulo $p$ which is defined as
\begin{equation*}
    \sum_{n=1}^{p-1} \paranthesis{\frac{n}{p}} e(an/p) = \sum_{n=0}^{p-1} e(an^2/p)
\end{equation*}
with $\paranthesis{\frac{n}{p}}$ being the Legendre symbol. We then define
\begin{equation} \label{defn:S(h)}
    \mathcal{S}(h) = \llprod_{p > 2} \paranthesis{1+\frac{\mathcal{B}(p,h)}{(p-1)^4}}
\end{equation}
We also define
\begin{equation} \label{defn:r(h)}
    r_l(h) = \# \left\{ \{(u_j,v_j)\}_{1 \le j \le l} \, : \, \sum_{j = 1}^l (2^{u_j}-2^{v_j}) = h, 4 \le u_j,v_j \le L \right\}.
\end{equation}
Using this notation, we can now state the following result.
\begin{lemma} \label{lem:minorarcs}
For $l \ge 5$,
\begin{equation*}
\int_{\mathfrak{m}} \abs{{T(\alpha)}^4 G(\alpha)^{2l}} d\alpha \le 8(11+\epsilon)\paranthesis{1+O(\eta)}\mathfrak{J}(0) N \sum_{h \neq 0} r_l(h) \mathcal{S}(h) + O \paranthesis{NL^{2l-1+\epsilon}}.
\end{equation*}
\end{lemma}
The proof of this lemma is similar to that of Lemma 3.3 in \cite{zhao2014four}, which corresponds to the case $l = 7$. The factor $(11+\epsilon)$ comes from using Rosser's weight of order $D = N^{1/12-\epsilon}$, as in \cite{tsang2017on}. For a complete proof, see Lemma 4.4.2 in \cite{zhao2012some}.

Therefore, to estimate the integral over the minor arcs, we must first estimate the quantity $\mathcal{S}(h)$ and then the sum $\sum_{h \neq 0} r_l(h) \mathcal{S}(h)$.

We know that $g(a;p)=\paranthesis{\frac{a}{p}}g(1;p)$. The Gauss sum $g(1;p)$ evaluates to $\sqrt{p}$ when $p \equiv 1 \mod 4$ and $i \sqrt{p}$ when $p \equiv 3 \mod 4$. Substituting these values, we obtain
\begin{equation*}
    \mathcal{B}(p,h) = \left\{
        \begin{array}{ll}
        -(p+1)^2 & \text{if } p \equiv 3 \mod 4 \text{ and } p \nmid h, \\
        -(p^2+6p+1)-4p(p+1)\paranthesis{\frac{h}{p}} & \text{if } p \equiv 1 \mod 4 \text{ and } p \nmid h, \\
        (p-1)(p+1)^2 & \text{if } p \equiv 3 \mod 4 \text{ and } p \mid h, \\
        (p-1)(p^2+6p+1) & \text{if } p \equiv 1 \mod 4 \text{ and } p \mid h.
        \end{array}
    \right.
\end{equation*}

We would like to estimate $\mathcal{S}(h)$ in terms of the functions $a$ and $b$, given by
\begin{align*}
    a(p) & = \left\{
        \begin{array}{ll}
        -(p+1)^2 & \text{if } p \equiv 3 \mod 4, \\
        3p^2-2p-1 & \text{if } p \equiv 1 \mod 4,
        \end{array}
    \right. \\
    b(p) & = \left\{
        \begin{array}{ll}
        (p-1)(p+1)^2 & \text{if } p \equiv 3 \mod 4, \\
        (p-1)(p^2+6p+1) & \text{if } p \equiv 1 \mod 4.
        \end{array}
    \right.
\end{align*}
Note that $a(p)$ is an upper bound on the value of $\mathcal{B}(p,h)$ when $p\nmid h$. Now, suppose $\mathcal{P}$ is a small subset of odd primes, then we have
\begin{align*}
    \mathcal{S}(h) & \le \llprod_{p \in \mathcal{P}} \paranthesis{1+\frac{\mathcal{B}(p,h)}{(p-1)^4}} \llprod_{\substack{p \not\in \mathcal{P} \\ p \nmid h}} \paranthesis{1+\frac{a(p)}{(p-1)^4}} \llprod_{\substack{p \not\in \mathcal{P} \\ p \mid h}} \paranthesis{1+\frac{b(p)}{(p-1)^4}} \\
    & \le \llprod_{p \in \mathcal{P}} \paranthesis{1+\frac{\mathcal{B}(p,h)}{(p-1)^4}} \llprod_{p \not\in \mathcal{P}} \paranthesis{1+\frac{a(p)}{(p-1)^4}} \llprod_{\substack{p \not\in \mathcal{P} \\ p \mid h}} \frac{\paranthesis{1+\frac{b(p)}{(p-1)^4}}}{\paranthesis{1+\frac{a(p)}{(p-1)^4}}}.
\end{align*}
Set $\kappa(\mathcal{P},h)$ to $\prod_{p \in \mathcal{P}} \paranthesis{1+\frac{\mathcal{B}(p,h)}{(p-1)^4}}$. Let $c(d)$ be a multiplicative function for $d$ square-free and coprime to the product of primes $2 \prod_{p \in \mathcal{P}} p$, that is defined on primes as:
\begin{equation*}
    1+\invert{c(p)} = \frac{1+\frac{b(p)}{(p-1)^4}}{1+\frac{a(p)}{(p-1)^4}}.
\end{equation*}
Finally, let
\begin{equation*}
    c_{4,\mathcal{P}} = \llprod_{p \not\in \mathcal{P}} \paranthesis{1+\frac{a(p)}{(p-1)^4}}.
\end{equation*}
If we choose $\mathcal{P} = \{3,5\}$, then $c_{4,\mathcal{P}} \le 0.9743$. To show this, we note that,
\begin{equation*}
    1+\frac{a(p)}{(p-1)^4} \le 1+\frac{3.1}{p^2} \le \paranthesis{1-\invert{p^2}}^{-3.1}
\end{equation*}
for $p \ge 103$. Computing the product $\prod_{p > 5} \paranthesis{1+\frac{a(p)}{(p-1)^4}}$ for primes up to 100000 and bounding the tail by comparing with $\zeta(2)$, we have
\begin{align*}
    c_{4,\mathcal{P}} & \le \llprod_{5 < p < 100000} \paranthesis{1+\frac{a(p)}{(p-1)^4}} \llprod_{p \ge 100000} \paranthesis{1-\invert{p^2}}^{-3.1} \\
    & \le 0.97425 \times \zeta(2)^{3.1} \times \llprod_{p < 100000} \paranthesis{1-\invert{p^2}}^{3.1} \\
    & \le 0.9743.
\end{align*}
Since $13 \equiv 1 \mod 4$, adding it to the set $\mathcal{P}$ will result in a slight improvement in the upper bound on $\mathcal{S}(h)$ but for our purpose, the aforementioned $\mathcal{P}$ suffices. For brevity, denote $c_{4,\mathcal{P}}$ by $c_4$ and $\kappa(\{3,5\},h)$ by $\kappa(h)$, then
\begin{equation} \label{defn:kappah}
    \kappa(h) = \left\{
        \begin{array}{ll}
        0 & \text{if } 3\nmid h,\\
        \frac{15\paranthesis{5-3\paranthesis{\frac{h}{5}}}}{32} & \text{if } 3\mid h \text{ and } 5\nmid h, \\
        \frac{45}{8} & \text{if } 15\mid h,
        \end{array}
    \right.
\end{equation}
and hence,
\begin{equation} \label{ineq:S(h)}
    \mathcal{S}(h) \le c_4 \kappa(h) \llprod_{\substack{p > 5 \\ p \mid h}} \paranthesis{1+\invert{c(p)}}.
\end{equation}
\begin{remark} \label{rem:kappahmistake}
The expression for $\kappa(h)$ in (\ref{defn:kappah}) is different compared to Lemma 4.2 in \cite{zhao2014four}. This is due to two reasons: one, there in an error in Lemma 4.2 in \cite{zhao2014four} when $5 \nmid h$, and two, $\kappa(h)$ in (\ref{defn:kappah}) corresponds to the quantity $3\tilde{\kappa}(h)$ in the proof of Lemma 4.3 in \cite{zhao2014four}.
\end{remark}
Next, we estimate the sum $\sum_{h \neq 0} r_l(h) \mathcal{S}(h)$. From (\ref{ineq:S(h)}),
\begin{equation} \label{ineq:r(h)S(h)}
    \sum_{h \neq 0} r_l(h) \mathcal{S}(h) \le c_4 \llsum_{h \neq 0} r_l(h) \kappa(h) \llprod_{\substack{p > 5 \\ p \mid h}} \paranthesis{1+\invert{c(p)}}.
\end{equation}
To rewrite the sum in a more convenient form, we will introduce some notation. Let $\rho(q)$ denote the multiplicative order of $2$ (if it exists) modulo $q$.
Define
\begin{equation} \label{defn:Ndhl}
    N_l(d) = \#\left\{\{(u_j,v_j)\}_{1 \le j \le l} \,:\, d \mid \sum_{1 \le j \le l} (2^{u_j}-2^{v_j}), 1 \le u_j,v_j \le \rho(d)\right\}
\end{equation}
and
\begin{equation} \label{defn:beta}
    \beta_l(d) = \frac{\rho^{2l}(d)}{N_l(d)}.
\end{equation}

We also let
\begin{align*}
    c_{1,l} & := \sum_{p \mid d \Rightarrow p > 5} \frac{\mu^2(d)}{c(d)\beta_l(3d)}, \\
    c_{2,l} & := \sum_{p \mid d \Rightarrow p > 5} \frac{\mu^2(d)}{c(d)\beta_l(15d)}.
\end{align*}

\begin{lemma} \label{lem:sumoverkappah}
Let $\mathcal{S}(h)$ be given by (\ref{defn:S(h)}) and $r_l(h)$ be given by (\ref{defn:r(h)}). Then
\begin{equation} \label{ineq:c1c2}
    \sum_{h \neq 0} r_l(h) \mathcal{S}(h) \le c_{0,l} L^{2l}
\end{equation}
where
\begin{equation} \label{defn:c0}
    c_{0,l} := \frac{75}{32}c_{1,l}+\frac{105}{32}c_{2,l}.
\end{equation}
\end{lemma}
\begin{proof}
From (\ref{ineq:r(h)S(h)}), we can rewrite the left-hand side of (\ref{ineq:c1c2}) as
\begin{equation*}
    \frac{15}{32} c_4 \llsum_{\substack{h \neq 0 \\ 3 \mid h,\, 5 \nmid h}} \paranthesis{5-3\paranthesis{\frac{h}{5}}} r_l(h) \llprod_{\substack{p > 5 \\ p \mid h}} \paranthesis{1+\invert{c(p)}} + \frac{45}{8} c_4 \llsum_{\substack{h \neq 0 \\ 15 \mid h}} r_l(h) \llprod_{\substack{p > 5 \\ p \mid h}} \paranthesis{1+\invert{c(p)}}.
\end{equation*}
We must first show that the term with the $\paranthesis{\frac{h}{5}}$ factor is $o\paranthesis{L^{2l}}$. For $p \ge 5$,
\begin{equation*}
    1+\invert{c(p)} \le 1+\frac{4.1}{p} \le \paranthesis{1-\invert{p}}^{-4.1}.
\end{equation*}
Therefore,
\begin{equation*}
    \llprod_{\substack{p > 5 \\ p \mid h}} \paranthesis{1+\invert{c(p)}} \le \paranthesis{\frac{h}{\phi(h)}}^{4.1}.
\end{equation*}
Upon noting that $h \le 2^{L+3}$ and bounding $h/\phi(h)$ by (3.42) of \cite{rosser1962approximate}, we have
\begin{equation*}
    \llprod_{\substack{p > 5 \\ p \mid h}} \paranthesis{1+\invert{c(p)}} = O \paranthesis{(\log L)^{4.1}} = o(L).
\end{equation*}
We can now\footnote{The author thanks Lilu Zhao for suggesting this part of the proof.} compare the term with the $\paranthesis{\frac{h}{5}}$ factor,
\begin{equation*}
    S = \frac{45}{32} c_4 \llsum_{\substack{h \neq 0 \\ 3 \mid h,\, 5 \nmid h}} \paranthesis{\frac{h}{5}} r_l(h) \llprod_{\substack{p > 5 \\ p \mid h}} \paranthesis{1+\invert{c(p)}},
\end{equation*}
with
\begin{align*}
    S' &= \frac{45}{32} c_4 \llsum_{\substack{h \neq 0 \\ 3 \mid h,\, 5 \nmid h}} \paranthesis{\frac{h}{5}} r_l'(h) \llprod_{\substack{p > 5 \\ p \mid h}} \paranthesis{1+\invert{c(p)}} \text{ and} \\ 
    S'' & = \frac{45}{32} c_4 \llsum_{\substack{h \neq 0 \\ 3 \mid h,\, 5 \nmid h}} \paranthesis{\frac{h}{5}} r_l''(h) \llprod_{\substack{p > 5 \\ p \mid h}} \paranthesis{1+\invert{c(p)}},
\end{align*}
where
\begin{align*}
    r_l'(h) := \# \left\{ \{(u_j,v_j)\}_{1 \le j \le l} \, : \, \sum_{j = 1}^l (2^{u_j}-2^{v_j}) = h, 5 \le u_j,v_j \le L \right\} \text{ and} \\
    r_l''(h) := \# \left\{ \{(u_j,v_j)\}_{1 \le j \le l} \, : \, \sum_{j = 1}^l (2^{u_j}-2^{v_j}) = h, 4 \le u_j,v_j \le L-1 \right\}.
\end{align*}
We claim that $S-S' = o(L^{2l})$ and $S-S'' = o(L^{2l})$. Clearly, any value of $h$ that contributes to the sum $S-S'$ has a representation such that at least one of $u_j,v_j$ ($1 \le j \le l$) is 4. This fixes one of the $2l$ variables, whence, $r_l'(h) = O(L^{2l-1})$. Taking into account the contribution from the product over $1+\invert{c(p)}$, we obtain $S-S' = o(L^{2l})$, as claimed. A similar argument also proves $S-S'' = o(L^{2l})$. There is a one-to-one correspondence between vectors $(u_1,\ldots,u_l,v_1,\ldots,v_l)$ that contribute to the sum $S'$ and vectors that contribute to the sum $S''$ which can be given by the map
\begin{equation*}
    (u_1,\ldots,u_l,v_1,\ldots,v_l) \mapsto (u_1-1,\ldots,u_l-1,v_1-1,\ldots,v_l-1).
\end{equation*}
Such a map will change $h$ by a factor of 2, and in particular, $S' = \paranthesis{\frac{2}{5}} S'' = -S''$, which implies, $S = o(L^{2l})$.

Therefore, we have reduced the left-hand side of (\ref{ineq:r(h)S(h)}) to
\begin{align*}
    & \frac{75}{32} c_4 \llsum_{\substack{h \neq 0 \\ 3 \mid h,\, 5 \nmid h}} r_l(h) \llprod_{\substack{p > 5 \\ p \mid h}} \paranthesis{1+\invert{c(p)}} + \frac{45}{8} c_4 \llsum_{\substack{h \neq 0 \\ 15 \mid h}} r_l(h) \llprod_{\substack{p > 5 \\ p \mid h}} \paranthesis{1+\invert{c(p)}} + o(L^{2l}) \\
    &= \frac{75}{32} c_4 \llsum_{\substack{h \neq 0 \\ 3 \mid h}} r_l(h) \llprod_{\substack{p > 5 \\ p \mid h}} \paranthesis{1+\invert{c(p)}} + \frac{105}{32} c_4 \llsum_{\substack{h \neq 0 \\ 15 \mid h}} r_l(h) \llprod_{\substack{p > 5 \\ p \mid h}} \paranthesis{1+\invert{c(p)}} + o(L^{2l}) \\
    &=: \frac{75}{32} c_4 \Sigma_1 + \frac{105}{32} c_4 \Sigma_2 + o(L^{2l}).
\end{align*}
As in the proof of Lemma 4.2 in \cite{zhao2014four}, we have
\begin{align*}
    \Sigma_1 &= \sum_{\substack{h \neq 0 \\ 3 \mid h}} r_l(h) \sum_{\substack{d|h \\ p \mid d \Rightarrow p > 5}} \frac{\mu^2(d)}{c(d)} \\
    &\le \sum_{\substack{d < N^\epsilon \\ p \mid d \Rightarrow p > 5}} \frac{\mu^2(d)}{c(d)} \sum_{\substack{1 \le u_j,v_j \le L \\ 3d|h}} 1 + O \paranthesis{N^{-\epsilon}} =: \Sigma_1' + O \paranthesis{N^{-\epsilon}}.
\end{align*}
Here, we have used the fact that $\invert{c(d)} \ll \invert{d}$. We split the sum $\Sigma_1'$ into two parts depending upon whether $\rho(3d) < L$ or $\rho(3d) \ge L$. In the latter case, the number of possible values of $d$ are bounded above by $\epsilon L$. Hence, we have
\begin{align*}
    \Sigma_1' &\le \sum_{\substack{d < N^\epsilon \\ p \mid d \Rightarrow p > 5 \\ \rho(3d) < L}} \frac{\mu^2(d)}{c(d)} \sum_{\substack{1 \le u_j,v_j \le \rho(3d) \\ 3d|h}} \paranthesis{\frac{L}{\rho(3d)}+O(1)}^{2l} + \sum_{\substack{d < N^\epsilon \\ p \mid d \Rightarrow p > 5 \\ \rho(3d) \ge L}} \frac{\mu^2(d)}{c(d)} L^{2l-1} \\
    &\le L^{2l} \sum_{p \mid d \Rightarrow p > 5} \frac{\mu^2(d)}{c(d)\rho^{2l}(3d)} N_l(3d) + \epsilon L^{2l}.
\end{align*}
From this, we can conclude that $\Sigma_1 \le (c_{1,l}+\epsilon)L^{2l}$. Similarly, $\Sigma_2 \le (c_{2,l}+\epsilon)L^{2l}$. Since $c_4 < 1$, we can choose an appropriate $\epsilon$ to obtain the lemma.

\end{proof}


\section{Computational work} \label{sec:numcomp}

In this section, we will outline the numerical computations for estimating the value of $c_{0,l}$, as defined in (\ref{defn:c0}), for $1 \le l \le 8$. 

\begin{remark}
\label{rem:c0isdifferent}
As noted in Remark \ref{rem:kappahmistake} for $\kappa(h)$, the value of $c_0$ in Lemma 4.1 of \cite{zhao2014four} corresponds to $\frac{c_{0,7}}{3}$ in Lemma \ref{lem:c0}.
\end{remark}

Let $N_l(d)$ be defined as in (\ref{defn:Ndhl}). It can be interpreted as the number of solutions of the equation
\begin{equation} \label{eqn:betacongruence}
    \sum_{\substack{1 \le u_j \le \rho(d) \\ 1 \le j \le l}} 2^{u_j} \equiv \sum_{\substack{1 \le v_j \le \rho(d) \\ 1 \le j \le l}} 2^{v_j} \mod d.
\end{equation}
The maximum possible number of solutions of (\ref{eqn:betacongruence}) is $\rho^{2l-1}(d)$ since fixing $2l-1$ powers of 2 out of $2l$ possible powers in (\ref{eqn:betacongruence}) essentially determines the value that the $2l$-th power can take (which may or may not be permissible). From the definition of $\beta_l(d)$ in (\ref{defn:beta}), this implies that $\beta_l(d) \ge \rho(d)$. Further, since $2^{\rho(d)} \ge d+1$, we have
\begin{equation} \label{ineq:betarho}
    \beta_l(d) \ge \rho(d) \ge \log(d+1)/\log 2.
\end{equation}
We also note that if $q_1|q$, then $\rho(q_1)|\rho(q)$ and so,
\begin{equation} \label{eqn:rhoproperty}
    \text{lcm}\left\{\rho(p): p \text{ is prime and } p|q \right\} | \rho(q).
\end{equation}
Similarly, if $d' | d$, then $\rho(d) = t \rho(d')$, for some positive integer $t$. Every solution of (\ref{eqn:betacongruence}) can be reduced modulo $d'$ but different solutions might lead to the same reduced solution. Again, fixing $2l-1$ powers of 2, there can be at most $t$ solutions of (\ref{eqn:betacongruence}) that a reduced solution corresponds to. Hence, $N_l(d') \ge N_l(d)/t^{2l}$. In particular, if $d' | d$,
\begin{equation} \label{ineq:betarho2}
    \rho(d') \le \rho(d) \;\; \text{ and } \;\; \beta_l(d') \le \beta_l(d).
\end{equation}
We also have that $\rho^2(d) N_{l-1}(d) \ge N_l(d)\ge \rho(d) N_{l-1}(d)$ since given a solution of (\ref{eqn:betacongruence}) for $l-1$, we can construct at least $\rho(d)$ number of solutions (by adding the same power of 2 on both sides) and at most $\rho^2(d)$ number of solutions corresponding to $l$. Therefore,
\begin{equation} \label{ineq:betal}
    \beta_{l}(d) \ge \beta_{l-1}(d) \ge \beta_l(d)/\rho(d). 
\end{equation}
Using these ideas, we can substantially reduce computational requirements for estimating $c_{0,l}$ and obtain the following result.
\begin{lemma} \label{lem:c0}
We have
\begin{align*}
    c_{0,1} < 1.127,\hphantom{00} \;\; c_{0,2} &< 0.803, \\
    c_{0,3} < 0.782,\hphantom{00} \;\; c_{0,4} &< 0.779, \\
    c_{0,5} < 0.7773, \hphantom{0} \;\; c_{0,6} &< 0.7772, \\
    c_{0,7} < 0.77708, \;\; c_{0,8} &< 0.77707.
\end{align*}
\end{lemma}
\begin{proof}
The proof of the above lemma is similar to the proof of Lemma 4.1 in \cite{zhao2014four} but with a few modifications and a new set of values obtained through numerical computation. We can rewrite $c_{1,l}$ as
\begin{equation*}
    c_{1,l} = \sum_{p|d \Rightarrow p>5} \frac{\mu^2(d)}{c(d)} \int_{\beta_l(3d)}^{\infty} \frac{dx}{x^2} = \int_2^\infty \sum_{\substack{p|d \Rightarrow p>5 \\ \beta_l(3d) \le x}} \frac{\mu^2(d)}{c(d)} \frac{dx}{x^2}.
\end{equation*}
From (\ref{ineq:betarho}), we have the inequality,
\begin{equation*}
    \sum_{\substack{p|d \Rightarrow p>5 \\ \beta_l(3d) \le x}} \frac{\mu^2(d)}{c(d)} \le \sum_{\substack{p|d \Rightarrow p>5 \\ \rho(3d) \le x}} \frac{\mu^2(d)}{c(d)}.
\end{equation*}
If we define $m_1(x) = \prod_{e\le x/2} \paranthesis{2^{2e}-1}$, the condition $\rho(3d) \le x$ implies $3d|m_1(x)$ ($\rho(3d)$ must be even). Upon noting that $c(p) < (p-1)$ for all primes, we have, for $x \ge 2$,
\begin{align*}
    \sum_{\substack{p|d \Rightarrow p>5 \\ \beta_l(3d) \le x}} \frac{\mu^2(d)}{c(d)} \le \sum_{\substack{p|d \Rightarrow p>5 \\ 3d | m(x)}} \frac{\mu^2(d)}{c(d)} & \le \llprod_{\substack{p>5 \\ p | m(x)}} \paranthesis{1 + \invert{c(p)}} \\
    & \le \llprod_{p>5} \frac{1 + \invert{c(p)}}{1+\invert{p-1}} \llprod_{\substack{p>5 \\ p | m(x)}} \paranthesis{1 + \invert{p-1}}.
\end{align*}
Modifying the proof of Lemma 5 in \cite{liu1999on} slightly for $m_1(x)$, we can show that
\begin{equation} \label{ineq:LiuLiuWangImproved}
    \frac{m_1(x)}{\phi(m_1(x))} \le e^\gamma \log \frac{x}{2}
\end{equation}
for $x \ge 37$. Indeed, by (3.42) of \cite{rosser1962approximate}, for $d \ge 3$,
\begin{equation*}
    \frac{d}{\phi(d)} \le u(d),
\end{equation*}
where $u(d)$ is an increasing function for $d \ge 30$ given by
\begin{equation*}
    u(d) := e^\gamma \log\log d + \frac{2.5064}{\log \log d}.
\end{equation*}
Here, $\gamma$ is the Euler--Mascheroni constant. We assume that $x \ge 37$. Since $m_1(x) \ge 30$ for $x \ge 4$, $u$ is an increasing function for our purposes. We have
\begin{align*}
    \frac{2m_1(x)}{\phi(2m_1(x))} & \le u(2m_1(x)) \\ 
    & \le u \paranthesis{2^{(x/2)^2+(x/2)+1}} \\
    & \le u \paranthesis{2^{1.06(x/2)^2}}.
\end{align*}
From the definition,
\begin{align*}
    u \paranthesis{2^{1.06(x/2)^2}} & \le 2 e^\gamma \log(x/2) + e^\gamma \paranthesis{\log (1.06 \log 2)} + \frac{2.5064}{2\log 18 +\log (1.06 \log 2)} \\
    & \le 2 e^\gamma \log(x/2).
\end{align*}
Noting that $2\nmid m_1(x)$, this proves our original assertion (\ref{ineq:LiuLiuWangImproved}). For $37 \le x \le 99$, we can obtain a slight further improvement by observing that the primes 101, 107, and 131 do not divide $m_1(x)$ ($\rho(p) > 99$ for each of these primes):
\begin{equation*}
    \frac{2Pm_1(x)}{\phi(2Pm_1(x))} \le u(2^{1.12(x/2)^2}),
\end{equation*}
where $P = 101\cdot107\cdot131$. As before, this implies
\begin{equation*} \label{ineq:TrudgianTrick}
    \frac{m_1(x)}{\phi(m_1(x))} \le \paranthesis{1-\invert{101}} \paranthesis{1-\invert{107}} \paranthesis{1-\invert{131}} e^\gamma \log \frac{x}{2} =: c_3' e^\gamma \log \frac{x}{2}
\end{equation*}
and so, for $37 \le x \le 99$,
\begin{equation} \label{ineq:forinexact}
    \sum_{\substack{p|d \Rightarrow p>5 \\ \beta_l(3d) \le x}} \frac{\mu^2(d)}{c(d)} \le \frac{8c_3 c_3'}{15} e^\gamma \log \frac{x}{2},
\end{equation}
where $c_3 = \prod_{p>5} \frac{1 + \invert{c(p)}}{1+\invert{p-1}}$. The factor $8/15$ comes from the fact that we only consider primes greater than 5 in the product, and $c_3' < 0.97336$, a slight improvement compared to 1. For $x > 99$,
\begin{equation} \label{ineq:forinexactgen}
    \sum_{\substack{p|d \Rightarrow p>5 \\ \beta_l(3d) \le x}} \frac{\mu^2(d)}{c(d)} \le \frac{8c_3}{15} e^\gamma \log \frac{x}{2}.
\end{equation}
\begin{remark}
Although the improvement in the [37,99] range is not significant in this case, we have included this argument here as it may be useful in similar problems where \cite{liu1999on}*{Lemma 5} is used.
\end{remark}

To bound $c_3$, we use a similar method as the one used for bounding $1+\invert{c(p)}$ in the proof of Lemma \ref{lem:sumoverkappah}. We write $\frac{1 + \invert{c(p)}}{1+\invert{p-1}}$ as $1+\varepsilon(p)$, and observe that
\begin{equation} \label{ineq:comparewithzeta}
    1+\varepsilon(p) \le  1+\frac{7.44}{p^2} \le \paranthesis{1-\invert{p^2}}^{-7.44}
\end{equation}
for $p \ge 742$. Computing the product $\prod_{p>5} \paranthesis{1+\varepsilon(p)}$ for primes up to $p_{1000000}$ (the one-millionth prime), and using (\ref{ineq:comparewithzeta}) along with the Euler product formula for the Riemann zeta-function to bound the tail, we obtain
\begin{align*}
    c_3 & \le \prod_{5<p \le p_{1000000}} \paranthesis{1+\varepsilon(p)} \prod_{p>p_{1000000}} \paranthesis{1+\varepsilon(p)} \\
    & \le 1.390399 \times \zeta(2)^{7.44} \times \llprod_{p \le p_{1000000}} \paranthesis{1-\invert{p^2}}^{7.44} \\
    & \le 1.3904.
\end{align*}

Now we can estimate $c_{1,l}$. We split $c_{1,l}$ into an exact and an inexact part, where the exact part is estimated using algorithms described below and the inexact part is estimated using (\ref{ineq:forinexact}) and (\ref{ineq:forinexactgen}).
\begin{align}
    c_{1,l} &= \int_2^M \sum_{\substack{p|d \Rightarrow p>5 \\ \beta_l(3d) \le x}} \frac{\mu^2(d)}{c(d)} \frac{dx}{x^2} + \int_M^\infty \sum_{\substack{p|d \Rightarrow p>5 \\ \beta_l(3d) \le x}} \frac{\mu^2(d)}{c(d)} \frac{dx}{x^2} \nonumber \\
    & \le \sum_{\substack{p|d \Rightarrow p>5 \\ d < (2^M-1)/3 \\ \beta_l(3d) < M}} \frac{\mu^2(d)}{c(d)} \paranthesis{\invert{\beta_l(3d)}-\invert{M}} \nonumber \\
    & \hphantom{\le} + \frac{8c_3}{15} e^\gamma \left\{ c_3' \paranthesis{\frac{1+\log(M/2)}{M} - \frac{1+\log(99/2)}{99}}+ \frac{1+\log(99/2)}{99} \right\}. \label{ineq:c1split}
\end{align}
Here, the inequality $d < (2^M-1)/3$ follows from (\ref{ineq:betarho}). The value of $M$ is chosen based on the height up to which $\beta_l(3d)$ is calculated, which in our case was $M = 37$. A very similar split can be done for $c_{2,l}$, albeit with $\beta_l(15d)$ and the inequality $d < (2^M-1)/15$. A larger denominator in the said inequality allows us to take $M = 39$. Adding $c_{1,l}$ and $c_{2,l}$ with the appropriate coefficients from (\ref{defn:c0}) proves Lemma \ref{lem:c0}.
\end{proof}


\subsection{Computation of $\rho(fd)$ and $\beta_l(fd)$}
\label{subsec:computations}

Since computing $\beta_l(fd)$ (where $f = 3 \text{ or } 15$) in (\ref{ineq:c1split}) is a time and memory intensive process, we only compute its value when absolutely necessary. From (\ref{ineq:betarho}), (\ref{eqn:rhoproperty}), and (\ref{ineq:betarho2}), we only need to compute $\beta_l(fd)$ for square-free numbers $d$ such that $\log_2(fd+1) < M$ and for all primes $p_1,\ldots,p_k|d$ such that $\text{lcm}\left\{ \rho(f),\rho(p_1),\ldots,\rho(p_k) \right\} < M_0$ ($f$ is coprime to $d$). Moreover, we only need to compute $\beta_l(15d)$ if $\beta_l(3d) < M$. To decide whether to compute $\beta_l(fd)$ for a particular $d$, we must first compute $\rho(fd)$. Although this is a computationally hard problem, we can bypass this issue by simply checking if $\rho(p) < M$ for all primes up to $(2^{M}-1)/f$. For this purpose, we used the list of the first 2 billion primes \cite{primenolist}, verifying the primality of the numbers and the value of the prime counting function at the relevant height using Pari/GP. Finally, due to (\ref{ineq:betal}), we only need to compute $\beta_l(fd)$ if $\beta_{l-1}(fd) < M$.

We use two different algorithms to compute $\beta_l(fd)$. Algorithm \ref{algo:beta1D} has space complexity $O(d)$ whereas Algorithm \ref{algo:beta2D} is more memory intensive and has space complexity $O(d\rho(fd))$. Therefore, for large $d$, it is not practical to use Algorithm \ref{algo:beta2D}. However, Algorithm \ref{algo:beta2D} could be easier to implement using standard libraries for sparse matrices (like SciPy in Python). We implemented both the algorithms (in Python) and were thus able to confirm the values of $\beta_l(fd)$ with different algorithms. Table \ref{tab:beta} gives the value of $\beta_7(3d)$ for a few large primes along with the time required to compute it using Algorithm \ref{algo:beta2D} on an Intel Core i7 processor with 2.60GHz clock speed and 62.5 GiB of memory.

\begin{table}[h]
\centering
\caption{$\beta_7(3d)$ for large primes}
\begin{tabular}{|c|c|c|}
\hline
$p$ & $\beta_7(3p)$ & \text{Time (in s)} \\
\hline
$22366891$ & $3089168.27$ & $15238$\\
\hline
$25781083$ & $15652237.24$ & $19914$\\
\hline
$164511353$ & $56626483.49$ & $13332$\\
\hline
$616318177$ & $28269951.69$ & $6728$ \\
\hline
\end{tabular}
\label{tab:beta}
\end{table}

\section{Proof of Theorem \ref{thm:main}} \label{sec:proofofmain}

We follow the method set out in \cite{liu2004four} and \S \ref{sec:applicatons} of \cite{zhao2014four}. We first prove that every sufficiently large even integer of the form $N \equiv 4 \mod 8$ can be represented as a sum of four squares of primes and 29 powers of two. Then, by adding at most 2 powers of two, we cover all even congruence classes modulo 8, allowing us to generalise the result to all sufficiently large even integers.

Suppose
\begin{equation} \label{defn:lambdaset}
    \varepsilon(\lambda) := \{ \alpha \in (0,1] \,:\, G(\alpha) \ge \lambda L \}.
\end{equation}
By Table 1(A) of \cite{platt2015linnik}, we can choose $\lambda = \lambda_0 = 0.8844473$ for this purpose.
Setting $\mathfrak{m}_1 = \mathfrak{m} \cap \varepsilon(\lambda_0)$ and $\mathfrak{m}_2 = \mathfrak{m} \setminus \mathfrak{m}_1$, as in \cite{liu2004four}, we have
\begin{align} \label{ineq:minorarcs}
    \mathlarger{\mathlarger{\mathlarger{|}}} \int_{\mathfrak{m}} T^4(\alpha) G^{k'}(\alpha) & e(-\alpha N) d\alpha \mathlarger{\mathlarger{\mathlarger{|}}} \nonumber \\
    & \le \int_{\mathfrak{m}_1} \abs{ T^4(\alpha) G^{k'}(\alpha) d\alpha } + \int_{\mathfrak{m}_2} \abs{ T^4(\alpha) G^{k'}(\alpha) d\alpha } \nonumber \\
    & \le O\paranthesis{N^{1-\epsilon}} + (\lambda_0 L)^{k'-2l} \int_0^1 \abs{T(\alpha)^4 G(\alpha)^{2l}} d\alpha.
\end{align}
We can now estimate the integral using Lemma \ref{lem:minorarcs} and Lemma \ref{lem:sumoverkappah}. For $l \ge 5$,
\begin{align*}
    \mathlarger{\mathlarger{\mathlarger{|}}} \int_{\mathfrak{m}} T^4(\alpha) G^{k'}(\alpha) & e(-\alpha N) d\alpha \mathlarger{\mathlarger{\mathlarger{|}}} \\
    & \le O(N) + (\lambda_0 L)^{k'-2l} c_{0,l} \times 8 (11+O(\eta)+\epsilon) \mathfrak{J}(0)NL^{2l}.
\end{align*}
The estimate on the major arcs is given by Lemma \ref{lem:majorarcs}. Combining the estimates on the major arcs and minor arcs, we obtain
\begin{equation} \label{ineq:Rk}
    R_{k'}(N) \ge 8NL^{k'} \paranthesis{0.9 \mathfrak{J}(1) - {\lambda_0}^{k'-2l}(11+\epsilon)c_{0,l}(1+O(\eta))\mathfrak{J}(0)}
\end{equation}
for sufficiently small $\eta$. We choose $l = 5$ (see Table \ref{tab:l-k}) and substitute the bound on $c_{0,5}$ from Lemma \ref{lem:c0} in (\ref{ineq:Rk}). Finally, using the inequalities
$\mathfrak{J}(1) \le \mathfrak{J}(0) \le (1+O(\eta))\mathfrak{J}(1)$ from \cite{zhao2014four}*{p. 259} in (\ref{ineq:Rk}), we prove that $R_{k'}(N) > 0$ for sufficiently large $N \equiv 4 \mod 8$ and $k'=29$. Therefore, for $k = k'+2 = 31$, we have $R_k(N) > 0$ for sufficiently large even $N$, thus establishing Theorem \ref{thm:main}.


\begin{table}[h]
\centering
\caption{Minimum $k'$ such that $R_{k'}(N) > 0$ for $5 \le l \le 8$}
\begin{tabular}{|c|c|c|c|c|c|c|c|c|c|}
\hline
$l$ & $5$ & $6$ & $7$ & $8$ \\
\hline
$k'$ & $29$ & $31$ & $33$ & $35$\\
\hline
\end{tabular}
\label{tab:l-k}
\end{table}

\subsection{Further improvements}
\label{subsec:further}

To further improve Theorem \ref{thm:main} via the strategy used in this paper, one could choose a higher value of $M$ in (\ref{ineq:c1split}) and carry out the necessary computations as described in \S \ref{subsec:computations}. However, any significant improvement to Lemma \ref{lem:c0} would require substantial amount of computational resources if the algorithms used are similar to the ones described in \S \ref{subsec:computations}. 

The other possibility is to prove Lemma \ref{lem:minorarcs} for $1 \le l \le 4$. Table \ref{hypotab:l-k} gives the number of powers of two in Theorem \ref{thm:main} assuming that Lemma \ref{lem:minorarcs} holds for $1 \le l \le 4$.

Although we have not attempted to improve the major arcs estimate here, improving the factor 0.9 in Lemma \ref{lem:majorarcs} can give us an overall improvement in Theorem \ref{thm:main}. In particular, proving that this factor $> 0.9379$ for $k \ge 28$ would allow us to conclude that 30 powers of two suffice. This could be done by choosing a higher value of $m_0$ in the proof of \cite[Lemma~4.4]{zhao2014four} but the computation time for calculating the new factor increases substantially as we consider higher values of $m_0$. Choosing $m_0 = 19$, for example, was not enough to improve the result since this factor came out to be $0.92\ldots$ and took approximately 1300 s to compute on an Intel Core i5-7500 processor with 3.40GHz clock speed. For $m_0 = 23$, the computation did not terminate after more than 10 hours. It is plausible, however, that with a better algorithm and parallel computation, one can improve this result.

\section{Proof of Theorem \ref{thm:application} and other applications}
\label{sec:applicatons}

\begin{proof}[Proof of Theorem \ref{thm:application}]
The proof is straightforward. We replace the estimate $c_0 < 0.69$ in Lemma 2.2 of \cite{guangshi2018on} by $c_{0,5}/3 < 0.7773/3$ from Lemma \ref{lem:c0} (see Remark \ref{rem:c0isdifferent}). Now, using the estimate for $R_k(N)$ at the end of \S 3 of \cite{guangshi2018on} with the new value of $c_0$, we find that $\mathfrak{K} = 12$ suffices.
\end{proof}

Many similar problems in this area can be improved, either using Lemma \ref{lem:c0} directly or by carrying out the kind of computations done in \S 3. We have not attempted to improve them here but by listing them, we hope to draw attention to them for future work on the subject. One set of problems is simultaneous Linnik-style approximation of two integers. More precisely, what is the upper bound on the value of $k$ such that we can write
\begin{align*}
    N_1 &= {p_1}^2+{p_2}^2+{p_3}^2+{p_4}^2+2^{\nu_1}+\cdots+2^{\nu_k},\\
    N_2 &= {p_5}^2+{p_6}^2+{p_7}^2+{p_8}^2+2^{\nu_1}+\cdots+2^{\nu_k},
\end{align*}
for every pair of sufficiently large positive even integers $N_1$ and $N_2$ satisfying $N_2 \gg N_1 > N_2$ and $N_1 \equiv N_2 \mod 24$? In \cite{hu2021apair}, Hu, Kong, and Liu prove that $k = 98$ suffices using the methods employed by Zhao \cite{zhao2014four} and Kong and Liu \cite{kong2017on} and in the process, improving upon \cites{liu2013on,hu2015on}. Similarly, what is the upper bound on the value of $\mathfrak{K}$ such that
\begin{align*}
    N_1 &= p_1+{p_2}^2+{p_3}^2+2^{\nu_1}+\cdots+2^{\nu_{\mathfrak{K}}},\\
    N_2 &= p_4+{p_5}^2+{p_6}^2+2^{\nu_1}+\cdots+2^{\nu_{\mathfrak{K}}},
\end{align*}
for every pair of sufficiently large positive odd integers $N_1$ and $N_2$ satisfying $N_2 \gg N_1 > N_2$? Hu \cite{hu2019apair} established $\mathfrak{K} = 82$, thereby improving upon \cites{liu2013onpairs,hu2016on}.

A second set of problems involve Diophantine approximation using primes and prime powers. These can be considered the real analogues of Linnik--Goldbach problems. For example, as a real analogue of the representation of a number as the sum of a prime and two squares of primes, we consider numbers of the form
\begin{equation*}
    \lambda_1 p_1 + \lambda_2 {p_2}^2 + \lambda_3 {p_3}^2 + \mu_1 2^{\nu_1}+\mu_2 2^{\nu_2} +\cdots+\mu_s 2^{\nu_s},
\end{equation*}

where $p_1,p_2,p_3$ are prime numbers, $\nu_1,\nu_2,\ldots,\nu_s$ are positive integers, and the coefficients $\lambda_1,\lambda_2,\lambda_3,\mu_1,\mu_2,\ldots,\mu_s$ are real numbers satisfying certain conditions. For a precise description of the problem, see Theorem 1.1 in \cite{liu2020diophantine} which sharpens the results obtained in \cites{li2005diophantine,languasco2012on} using techniques from \cite{zhao2014four}.

\begin{table}[h]
\centering
\caption{Conditional number of powers of two ($k$) in Theorem \ref{thm:main} for $1 \le l \le 4$}
\begin{tabular}{|c|c|c|c|c|c|c|c|c|c|}
\hline
$l$ & $1$ & $2$ & $3$ & $4$ \\
\hline
$k$ & $26$ & $25$ & $27$ & $29$\\
\hline
\end{tabular}
\label{hypotab:l-k}
\end{table}

\newpage
\section{Appendix}
\label{sec:appendix}

\begin{algorithm}[H] \label{algo:beta1D}
\caption{Computing $\beta$ using 1D vectors}
\small
\SetKwInOut{KwIn}{Input}
\SetKwInOut{KwOut}{Output}
\KwIn{A square-free integer $d$, a positive integer $l$, and a multiplication factor $f$.}
\KwOut{The value $\beta_l(fd)$.}
\SetKwFunction{ccs}{countCongruenceSoln}
\SetKwFunction{cgcs}{computeCombSum}
\SetKwProg{Def}{def}{:}{end}

\tcp{Computes sums of residues modulo $fd$ and their frequencies}

\Def{$\cgcs(resCtr1,resCtr2,fd)$}{
Initialise $combSum$ as an empty dictionary/hash-table \\
\For{every $(key1,value1)$ pair in $resCtr1$}{
\For{every $(key2,value2)$ pair in $resCtr2$}{
Set $k$ to $key1+key2 \mod fd$ \\
Set $v$ to $value1*value2$ \\
\eIf{$k$ is a key in $combSum$}{
Update the value corresponding to the key $k$ in $combSum$ by adding $v$ to it
}{
Insert the key-value pair $(k,v)$ in $combSum$
}
}
}
return $CombSum$
}

\tcc{A slightly more optimised version of $\cgcs$ can be written when $resCtr1$ and $resCtr2$ are identical.}

\tcp{Computes $N(fd,h_l)$, as defined in (\ref{defn:Ndhl})}

\Def{$\ccs(fd,l,residues)$}{
Initialise $resCtr$ with $(r,1)$ key-value pairs for every residue $r$ in $residues$ \\
Initialise the list $binComb$ with $resCtr$ as its first element \\
Set $n$ to the length of binary representation of $l$ \\
\For{$i=1,\ldots, n-1$}{
Update $resCtr$ to the result of $\cgcs(resCtr,resCtr,fd)$\\
Update $binComb$ by inserting $resCtr$ at the top
}
Initialise $genComb$ to $resCtr$ \\
\For{$i=1,\ldots, n-1$}{
Update $resCtr$ to the $i$-th element of $binComb$ \\
\If{$(i+1)$-th bit is 1}{
Update $genComb$ to the result of $\cgcs(genComb,resCtr,fd)$
}
}
Set $congruenceSoln$ to the sum of $v^2$ for every value $v$ in $genComb$ \\ 
return $congruenceSoln$
}

Set $\rho$ to the order of 2 in the residue class $\mathbb{Z}/fd\mathbb{Z}$ \\
Set the list $residues$ to $1,2,2^2,\ldots,2^{\rho-1}$ \\
Set $\beta$ to $\rho^{2l}/\ccs(fd,l,residues)$ \\
\KwRet{$\beta$}
\end{algorithm}

\begin{algorithm}[h] \label{algo:beta2D}
\caption{Computing $\beta$ using 2D sparse matrices}
\small
\SetKwInOut{KwIn}{Input}
\SetKwInOut{KwOut}{Output}
\KwIn{A square-free integer $d$, a positive integer $l$, and a multiplication factor $f$.}
\KwOut{The value $\beta_l(fd)$.}
\SetKwFunction{csm}{constructSparseMatrix}
\SetKwProg{Def}{def}{:}{end}

\tcc{Creates a sparse matrix whose $(i,j)$-th element
represents the number of ways to transition from residue $i$ to residue $j$
by adding exactly one power of 2}

\Def{$\csm(fd,residues)$}{
Initialise a sparse matrix $mat$ \\
\For{$i=0,\ldots,fd-1$}{
\For{every $residue$ in $residues$}{
Set $j$ to $i+residue \mod fd$ \\
Set the $(i,j)$-th element of $mat$ to 1 \\
}
}
return $mat$
}

Set $\rho$ to the order of 2 in the residue class $\mathbb{Z}/fd\mathbb{Z}$ \\
Set the list $residues$ to $1,2,2^2,\ldots,2^{\rho-1}$ \\
Set $mat$ to the result of $\csm(fd,residues)$ \\
Compute the $l$-th power of $mat$ \tcp{matrix product}
\tcc{The $l$-th power of the original transition matrix encodes the number of ways to transition from one residue to another with exactly $l$ powers of 2.}
Set $congruenceSoln$ to the sum of $e^2$ for every element $e$ in the first row of the $l$-th power of $mat$ \\
Set $\beta$ to $\rho^{2l}/congruenceSoln$ \\
\KwRet{$\beta$}
\end{algorithm}

\section*{Acknowledgements}
The author thanks Benjamin Kaehler for suggesting Algorithm \ref{algo:beta2D} as an alternative way to compute $\beta$ and thus, providing an independent method to verify the results. The author also thanks Tim Trudgian for suggesting the $c_3'$ improvement in (\ref{ineq:forinexact}) and for his other suggestions and guidance throughout the course of this project. Finally, the author would like to thank Lilu Zhao for the helpful email correspondence with them.

\bibliographystyle{amsplain}
\bibliography{ref.bib}

\begin{bibdiv}
\begin{biblist}

\bib{primenolist}{misc}{
       title={First 2 billion prime numbers},
        note={Available at \url{http://www.primos.mat.br/2T_en.html}},
}

\bib{brudernlagrange1994}{article}{
      author={Br\"{u}dern, J.},
      author={Fouvry, \'{E}.},
       title={Lagrange's four squares theorem with almost prime variables},
        date={1994},
        ISSN={0075-4102},
     journal={J. Reine Angew. Math.},
      volume={454},
       pages={59\ndash 96},
  url={https://doi-org.wwwproxy1.library.unsw.edu.au/10.1515/crll.1994.454.59},
      review={\MR{1288679}},
}

\bib{gallagherprimes1975}{article}{
      author={Gallagher, P.~X.},
       title={Primes and powers of {$2$}},
        date={1975},
        ISSN={0020-9910},
     journal={Invent. Math.},
      volume={29},
      number={2},
       pages={125\ndash 142},
  url={https://doi-org.wwwproxy1.library.unsw.edu.au/10.1007/BF01390190},
      review={\MR{379410}},
}

\bib{heath-brownintegers2002}{article}{
      author={Heath-Brown, D.~R.},
      author={Puchta, J.-C.},
       title={Integers represented as a sum of primes and powers of two},
        date={2002},
        ISSN={1093-6106},
     journal={Asian J. Math.},
      volume={6},
      number={3},
       pages={535\ndash 565},
  url={https://doi-org.wwwproxy1.library.unsw.edu.au/10.4310/AJM.2002.v6.n3.a7},
      review={\MR{1946346}},
}

\bib{hu2019apair}{article}{
      author={Hu, L.},
       title={A pair of equations in one prime, two prime squares and powers of
  2},
        date={2019},
        ISSN={0065-1036},
     journal={Acta Arith.},
      volume={191},
      number={2},
       pages={191\ndash 200},
         url={https://doi.org/10.4064/aa190121-10-4},
      review={\MR{4008640}},
}

\bib{hu2021apair}{article}{
      author={Hu, L.},
      author={Kong, Y.},
      author={Liu, Z.},
       title={A pair of equations in four prime squares and powers of 2},
        date={2021},
        ISSN={1382-4090},
     journal={Ramanujan J.},
      volume={54},
      number={1},
       pages={79\ndash 92},
         url={https://doi.org/10.1007/s11139-019-00171-y},
      review={\MR{4198746}},
}

\bib{hu2015on}{article}{
      author={Hu, L.},
      author={Liu, H.},
       title={On pairs of four prime squares and powers of two},
        date={2015},
        ISSN={0022-314X},
     journal={J. Number Theory},
      volume={147},
       pages={594\ndash 604},
  url={https://doi-org.wwwproxy1.library.unsw.edu.au/10.1016/j.jnt.2014.08.006},
      review={\MR{3276342}},
}

\bib{hu2016on}{article}{
      author={Hu, L.},
      author={Yang, L.},
       title={On pairs of equations in one prime, two prime squares and powers
  of 2},
        date={2016},
        ISSN={0022-314X},
     journal={J. Number Theory},
      volume={160},
       pages={451\ndash 462},
  url={https://doi-org.wwwproxy1.library.unsw.edu.au/10.1016/j.jnt.2015.09.005},
      review={\MR{3425216}},
}

\bib{huasome1938}{article}{
      author={Hua, L.},
       title={Some results in the additive prime-number theory},
        date={1938},
        ISSN={0033-5606},
     journal={Quart. J. Math. Oxford Ser. (2)},
      volume={9},
      number={1},
       pages={68\ndash 80},
  url={https://doi-org.wwwproxy1.library.unsw.edu.au/10.1093/qmath/os-9.1.68},
      review={\MR{3363459}},
}

\bib{kong2017on}{article}{
      author={Kong, Y.},
      author={Liu, Z.},
       title={On pairs of {G}oldbach-{L}innik equations},
        date={2017},
        ISSN={0004-9727},
     journal={Bull. Aust. Math. Soc.},
      volume={95},
      number={2},
       pages={199\ndash 208},
         url={https://doi.org/10.1017/S000497271600071X},
      review={\MR{3614942}},
}

\bib{languasco2012on}{article}{
      author={Languasco, A.},
      author={Settimi, V.},
       title={On a {D}iophantine problem with one prime, two squares of primes
  and {$s$} powers of two},
        date={2012},
        ISSN={0065-1036},
     journal={Acta Arith.},
      volume={154},
      number={4},
       pages={385\ndash 412},
         url={https://doi.org/10.4064/aa154-4-4},
      review={\MR{2949876}},
}

\bib{languasco2010on}{article}{
      author={Languasco, A.},
      author={Zaccagnini, A.},
       title={On a {D}iophantine problem with two primes and {$s$} powers of
  two},
        date={2010},
        ISSN={0065-1036},
     journal={Acta Arith.},
      volume={145},
      number={2},
       pages={193\ndash 208},
         url={https://doi.org/10.4064/aa145-2-7},
      review={\MR{2733083}},
}

\bib{lifour2006}{article}{
      author={Li, H.},
       title={Four prime squares and powers of 2},
        date={2006},
        ISSN={0065-1036},
     journal={Acta Arith.},
      volume={125},
      number={4},
       pages={383\ndash 391},
         url={https://doi-org.wwwproxy1.library.unsw.edu.au/10.4064/aa125-4-6},
      review={\MR{2271284}},
}

\bib{li2007representation}{article}{
      author={Li, H.},
       title={Representation of odd integers as the sum of one prime, two
  squares of primes and powers of 2},
        date={2007},
        ISSN={0065-1036},
     journal={Acta Arith.},
      volume={128},
      number={3},
       pages={223\ndash 233},
         url={https://doi.org/10.4064/aa128-3-3},
      review={\MR{2313991}},
}

\bib{li2005diophantine}{article}{
      author={Li, W.},
      author={Wang, T.},
       title={Diophantine approximation by a prime, squares of two primes and
  powers of two},
        date={2005},
        ISSN={1008-5513},
     journal={Pure Appl. Math. (Xi'an)},
      volume={21},
      number={4},
       pages={295\ndash 299},
      review={\MR{2228057}},
}

\bib{linnikprime1951}{article}{
      author={Linnik, Y.},
       title={Prime numbers and powers of two},
        date={1951},
     journal={Trudy Nat. Inst. Steklov.},
      volume={38},
       pages={152\ndash 169},
      review={\MR{0050618}},
}

\bib{linnikaddition1953}{article}{
      author={Linnik, Y.},
       title={Addition of prime numbers with powers of one and the same
  number},
        date={1953},
     journal={Mat. Sbornik N.S.},
      volume={32(74)},
       pages={3\ndash 60},
  url={https://doi-org.wwwproxy1.library.unsw.edu.au/10.1016/0370-1573(77)90003-5},
      review={\MR{0059938}},
}

\bib{liu2016two}{article}{
      author={Liu, H.},
       title={Two results on powers of two in {W}aring-{G}oldbach type
  problems},
        date={2016},
        ISSN={1793-0421},
     journal={Int. J. Number Theory},
      volume={12},
      number={7},
       pages={1813\ndash 1825},
         url={https://doi.org/10.1142/S1793042116501128},
      review={\MR{3544414}},
}

\bib{liu2020diophantine}{article}{
      author={Liu, H.},
       title={Diophantine approximation with one prime, two squares of primes
  and powers of two},
        date={2020},
        ISSN={1382-4090},
     journal={Ramanujan J.},
      volume={51},
      number={1},
       pages={85\ndash 97},
  url={https://doi-org.wwwproxy1.library.unsw.edu.au/10.1007/s11139-018-0121-9},
      review={\MR{4047160}},
}

\bib{liu2003on}{article}{
      author={Liu, J.},
       title={On {L}agrange's theorem with prime variables},
        date={2003},
        ISSN={0033-5606},
     journal={Q. J. Math.},
      volume={54},
      number={4},
       pages={453\ndash 462},
  url={https://doi-org.wwwproxy1.library.unsw.edu.au/10.1093/qjmath/54.4.453},
      review={\MR{2031178}},
}

\bib{liurepresentation2000}{article}{
      author={Liu, J.},
      author={Liu, M.-C.},
       title={Representation of even integers as sums of squares of primes and
  powers of 2},
        date={2000},
        ISSN={0022-314X},
     journal={J. Number Theory},
      volume={83},
      number={2},
       pages={202\ndash 225},
  url={https://doi-org.wwwproxy1.library.unsw.edu.au/10.1006/jnth.1999.2500},
      review={\MR{1772613}},
}

\bib{liu1999on}{incollection}{
      author={Liu, J.},
      author={Liu, M.-C.},
      author={Wang, T.},
       title={On the almost {G}oldbach problem of {L}innik},
        date={1999},
      volume={11},
       pages={133\ndash 147},
         url={http://jtnb.cedram.org/item?id=JTNB_1999__11_1_133_0},
        note={Les XX\`emes Journ\'{e}es Arithm\'{e}tiques (Limoges, 1997)},
      review={\MR{1730436}},
}

\bib{liusquares1999}{article}{
      author={Liu, J.},
      author={Liu, M.-C.},
      author={Zhan, T.},
       title={Squares of primes and powers of {$2$}},
        date={1999},
        ISSN={0026-9255},
     journal={Monatsh. Math.},
      volume={128},
      number={4},
       pages={283\ndash 313},
  url={https://doi-org.wwwproxy1.library.unsw.edu.au/10.1007/s006050050065},
      review={\MR{1726765}},
}

\bib{liu2004four}{article}{
      author={Liu, J.},
      author={L\"{u}, G.},
       title={Four squares of primes and 165 powers of 2},
        date={2004},
        ISSN={0065-1036},
     journal={Acta Arith.},
      volume={114},
      number={1},
       pages={55\ndash 70},
         url={https://doi-org.wwwproxy1.library.unsw.edu.au/10.4064/aa114-1-4},
      review={\MR{2067872}},
}

\bib{liu2004representation}{article}{
      author={Liu, T.},
       title={Representation of odd integers as the sum of one prime, two
  squares of primes and powers of 2},
        date={2004},
        ISSN={0065-1036},
     journal={Acta Arith.},
      volume={115},
      number={2},
       pages={97\ndash 118},
         url={https://doi.org/10.4064/aa115-2-1},
      review={\MR{2099833}},
}

\bib{liu2013onpairs}{article}{
      author={Liu, Z.},
       title={On pairs of one prime, two prime squares and powers of 2},
        date={2013},
        ISSN={1793-0421},
     journal={Int. J. Number Theory},
      volume={9},
      number={6},
       pages={1413\ndash 1421},
         url={https://doi.org/10.1142/S1793042113500346},
      review={\MR{3103895}},
}

\bib{liu2013on}{article}{
      author={Liu, Z.},
       title={On pairs of quadratic equations in primes and powers of 2},
        date={2013},
        ISSN={0022-314X},
     journal={J. Number Theory},
      volume={133},
      number={10},
       pages={3339\ndash 3347},
         url={https://doi.org/10.1016/j.jnt.2013.04.006},
      review={\MR{3071816}},
}

\bib{liu2014one}{article}{
      author={Liu, Z.},
       title={One prime, two squares of primes and powers of 2},
        date={2014},
        ISSN={0236-5294},
     journal={Acta Math. Hungar.},
      volume={143},
      number={1},
       pages={3\ndash 12},
  url={https://doi-org.wwwproxy1.library.unsw.edu.au/10.1007/s10474-014-0393-5},
      review={\MR{3215599}},
}

\bib{guangshi2018on}{article}{
      author={L\"{u}, G.},
       title={On sum of one prime, two squares of primes and powers of 2},
        date={2018},
        ISSN={0026-9255},
     journal={Monatsh. Math.},
      volume={187},
      number={1},
       pages={113\ndash 123},
         url={https://doi.org/10.1007/s00605-017-1104-4},
      review={\MR{3842938}},
}

\bib{lu2009integers}{article}{
      author={L\"{u}, G.},
      author={Sun, H.},
       title={Integers represented as the sum of one prime, two squares of
  primes and powers of 2},
        date={2009},
        ISSN={0002-9939},
     journal={Proc. Amer. Math. Soc.},
      volume={137},
      number={4},
       pages={1185\ndash 1191},
         url={https://doi.org/10.1090/S0002-9939-08-09603-2},
      review={\MR{2465639}},
}

\bib{pintzon2003}{article}{
      author={Pintz, J.},
      author={Ruzsa, I.~Z.},
       title={On {L}innik's approximation to {G}oldbach's problem. {I}},
        date={2003},
        ISSN={0065-1036},
     journal={Acta Arith.},
      volume={109},
      number={2},
       pages={169\ndash 194},
         url={https://doi-org.wwwproxy1.library.unsw.edu.au/10.4064/aa109-2-6},
      review={\MR{1980645}},
}

\bib{pintzon2020}{article}{
      author={Pintz, J.},
      author={Ruzsa, I.~Z.},
       title={On {L}innik's approximation to {G}oldbach's problem. {II}},
        date={2020},
        ISSN={0236-5294},
     journal={Acta Math. Hungar.},
      volume={161},
      number={2},
       pages={569\ndash 582},
  url={https://doi-org.wwwproxy1.library.unsw.edu.au/10.1007/s10474-020-01077-8},
      review={\MR{4131935}},
}

\bib{platt2015linnik}{article}{
      author={Platt, D.},
      author={Trudgian, T.},
       title={Linnik's approximation to {G}oldbach's conjecture, and other
  problems},
        date={2015},
        ISSN={0022-314X},
     journal={J. Number Theory},
      volume={153},
       pages={54\ndash 62},
  url={https://doi-org.wwwproxy1.library.unsw.edu.au/10.1016/j.jnt.2015.01.008},
      review={\MR{3327564}},
}

\bib{rosser1962approximate}{article}{
      author={Rosser, J.},
      author={Schoenfeld, L.},
       title={Approximate formulas for some functions of prime numbers},
        date={1962},
        ISSN={0019-2082},
     journal={Illinois J. Math.},
      volume={6},
       pages={64\ndash 94},
         url={http://projecteuclid.org/euclid.ijm/1255631807},
      review={\MR{137689}},
}

\bib{toleveon2005}{article}{
      author={Tolev, D.},
       title={On the exceptional set of {L}agrange's equation with three prime
  and one almost-prime variables},
        date={2005},
        ISSN={1246-7405},
     journal={J. Th\'{e}or. Nombres Bordeaux},
      volume={17},
      number={3},
       pages={925\ndash 948},
         url={http://jtnb.cedram.org/item?id=JTNB_2005__17_3_925_0},
      review={\MR{2212133}},
}

\bib{tsang2017on}{article}{
      author={Tsang, K.-M.},
      author={Zhao, L.},
       title={On {L}agrange's four squares theorem with almost prime
  variables},
        date={2017},
        ISSN={0075-4102},
     journal={J. Reine Angew. Math.},
      volume={726},
       pages={129\ndash 171},
  url={https://doi-org.wwwproxy1.library.unsw.edu.au/10.1515/crelle-2014-0094},
      review={\MR{3641655}},
}

\bib{wooleyslim2002}{article}{
      author={Wooley, T.},
       title={Slim exceptional sets for sums of four squares},
        date={2002},
        ISSN={0024-6115},
     journal={Proc. London Math. Soc. (3)},
      volume={85},
      number={1},
       pages={1\ndash 21},
  url={https://doi-org.wwwproxy1.library.unsw.edu.au/10.1112/S0024611502013515},
      review={\MR{1901366}},
}

\bib{zhao2012some}{article}{
      author={Zhao, L.},
       title={Some results on {W}aring--{G}oldbach type problems},
        date={2012},
     journal={HKU Theses Online (HKUTO)},
}

\bib{zhao2014four}{article}{
      author={Zhao, L.},
       title={Four squares of primes and powers of 2},
        date={2014},
        ISSN={0065-1036},
     journal={Acta Arith.},
      volume={162},
      number={3},
       pages={255\ndash 271},
         url={https://doi.org/10.4064/aa162-3-4},
      review={\MR{3173024}},
}

\end{biblist}
\end{bibdiv}
\end{document}